\documentclass[conference]{ieeeconf} 

\usepackage{graphicx}
\usepackage{graphics,epsfig,theorem,latexsym,amssymb,amsmath,picinpar,psfrag,subfigure}
\usepackage{epstopdf}
\usepackage{tikz}
\usepackage{color}
\usepackage{cite}
\usepackage{url}
\usepackage{ifthen}
\usepackage{amsmath}

\IEEEoverridecommandlockouts                              % This command is only
\def\dst{\displaystyle}

\def\Real{\mathbb R}

\def\dst{\displaystyle}

\def\bea{\begin{eqnarray}}
\def\eea{\end{eqnarray}}
\def\beann{\begin{eqnarray*}}
\def\eeann{\end{eqnarray*}}
\def\beeq#1{\begin{equation}{#1}\end{equation}}

\def\be{\begin{equation}}
\def\ee{\end{equation}}
\def\ba{\begin{array}}
\def\ea{\end{array}}
\def\bea{\begin{eqnarray}}
\def\eea{\end{eqnarray}}
\def\beann{\begin{eqnarray*}}
\def\eeann{\end{eqnarray*}}

\newtheorem{assumption}{Assumption}

\newtheorem{lemma}{Lemma}
\newtheorem{proposition}{Proposition}

\title{\LARGE \bf About synchronization of homogeneous nonlinear agents\\ over switching networks }

\author{G. Casadei, L. Marconi, and A. Isidori 
  \thanks{ Giacomo Casadei  and Lorenzo Marconi  are with C.A.SY. -- DEI,
 University of Bologna, Bologna, Italy  {\tt\small g.casadei@unibo.it}, {\tt\small lorenzo.marconi@unibo.it}.
 Research supported in part by the European Project SHERPA (G.A. 600958).  }
 \thanks{ A. Isidori is with DIS-``Sapienza" - Universit\`a di Roma, Roma, Italy {\tt \small albisidori@dis.uniroma1.it}}
 }
 
\begin{document}

 \maketitle
\thispagestyle{empty}
\pagestyle{empty}

\begin{abstract}
 
 In this technical note we address the problem of achieving consensus in a  network of homogeneous nonlinear systems. The communication network is supposed to be switching within a finite set of topologies which may be disconnected for finite time intervals. We prove that if the length of the time intervals in which connected topologies are active satisfy an average dwell-time condition consensus is achieved.  Lyapunov arguments proposed in the field of hybrid nonlinear systems are adopted to prove the result.
 
\end{abstract}

\section{Introduction}
 The problem of achieving consensus among a set of systems exchanging information through a network is extensively studied in the control literature (see \cite{Wieland-diss} for an extensive survey of results in the field). Information network analysis, multi-agent systems, electrical power systems, animal collective behaviour, systems biology are just a few applicative domains where consensus among networked agents plays a role.  
 One of the distinguishing elements of the framework where consensus problems are formulated is how the exchange of information between the agents is modelled. Graph theory is typically used in order to model communication networks (\cite{Fied1}) and conditions under which consensus can be achieved usually require the graph to be connected. A challenging context addressed in literature is when the graph is time-varying. In this field, among the different contributions presented in literature, we recall \cite{LiGuo}, in which the case of linear systems with switching topologies that are jointly connected is addressed, \cite{AlgoSwitch}, dealing with the case of switching networks with the graphs that are uniformly quasi-strongly connected and fulfilling a dwell-time condition, and   \cite{Moreau-b} considering the special case of integrators networked with time varying graphs that are connected in the average.  It is worth  also mentioning the contribution in \cite{Murray} in which switching graphs (with potentially disconnected topologies) and delays are considered with agents modelled as simple integrators. 
 
 In this work we consider the case in which the agents are homogenous nonlinear systems and the communication graph switches within a set of possible topologies not all necessarily connected and with time intervals in which connected topologies are active not necessarily fulfilling a dwell-time. More specifically, we show that consensus can be achieved provided that the time intervals in which a disconnected topology is active have a length that is upper bounded, while the time intervals in which the communication topologies are connected fulfil an {\it average dwell-time} condition (\cite{HespanhaMorse}).  Lyapunov arguments proposed in the field of hybrid control systems are used to prove the main result  (see \cite{SF}, \cite{CaiGT}). This work frames as an addendum of \cite{nostro} in which heterogenous systems networked with fixed connected topologies are considered. In particular, the arguments presented here can be used to show that the control framework of \cite{nostro} succeeds in achieving consensus between networked heterogeneous nonlinear systems 
even in presence of switching topologies of the kind considered in the following.

\section{Basic facts and problem statement}

The paper deals with the problem of reaching a consensus among a set of $N_{\rm a}$ agents, described by homogenous nonlinear systems, exchanging information over  a network whose communication topology switches within a finite set of $N_{\rm t}$ possible configurations. Denoting by ${\mathcal V}=\{v_1,v_2, \ldots, v_{N_{\rm a}}\}$ the set of $N_{\rm a}$ nodes of the network, and by ${\cal T}=\{ {\cal T}_1, \ldots, {\cal T}_{N_{\rm t}}\}$ the set of $N_{\rm t}$ communication topologies,  each topology ${\cal T}_i$, $i=1,\ldots, N_{\rm t}$ is described by  a {\em directed communication graph} given by the following objects: 
\begin{itemize}
\item ${\mathcal E}_i \subset {\mathcal V}\times {\mathcal V}$ is a set of {\em edges} that models the interconnection between nodes in the $i$-th topology, according to the following convention: $(v_k,v_j)$ belongs to ${\mathcal E}_i$ if there is a flow of  information from node $j$ to node $k$. It is assumed that there are no self-loops, i.e. that  $(v_k,v_k)\notin {\mathcal E}_i$.   
\item for each $(v_k,v_j) \in {\mathcal E}_i $ the flow of information from node $j$ to node $k$ in the $i$-th topology is {\em weighted} by the $(k,j)$-th entry $a^i_{kj} \geq 0$ of the so-called {\em adjacency matrix} $A_i \in \Real^{N_{\rm a}\times N_{\rm a}}$.
  \end{itemize}
  The {\em Laplacian matrix} $L_i$ associated to the $i$-th topology can be also immediately computed from the adjacency matrix $A_i$ in the usual way. In particular, $L_i$ is the  $N_{\rm a}\times N_{\rm a}$ real matrix whose $(k,j)$-th entry $\ell^i_{kj}$ is defined as 
 \[
  \ba{rclr}
 \ell^i_{kj}(t) &=& - a^i_{kj}&\mbox{for $k\ne j$}\\[2mm]
 \ell^i_{kj}(t) &=& \dst \sum_{m=1}^{N_{\rm a}} a_{km}^i &\mbox{for $k= j$}.\ea
 \]
By definition, the diagonal entries of $L_i$ are non-negative, the off-diagonal entries are non-positive and, for each row, the sum of all entries on this row is zero (namely $L_i$ is a \textit{Metzler} matrix).  

In the proposed framework, the set of topologies ${\cal T}$ describing the possible communication graphs might be also characterised by topologies that are not necessarily connected\footnote{We recall that a communication topology is said to be {\em connected} if there is a node $v$ from which any other node $v_k \in {\mathcal V}\setminus\{v\}$ can be reached, or equivalently if there is a path from $v$ to all $v_k$. In the previous definition  a {\em path} from node $v_j$ to node $v_k$ in the $i$-th topology is a sequence of  $r$ distinct nodes $\{v_{\ell_1},  \ldots, v_{\ell_r}\}$ with $v_{\ell_1}=v_j$ and $v_{\ell_r}=v_k$ such that $(v_{i+1},v_i)\in{\mathcal E}_i$}. For this reason we split the set $\cal T$ in two disjoint sets ${\cal T}_{\rm c}$ and ${\cal T}_{\rm nc}$, which fulfil ${\cal T}= {\cal T}_{\rm c} \cup {\cal T}_{\rm nc}$ and $ {\cal T}_{\rm c} \cap {\cal T}_{\rm nc} = \emptyset$,  collecting topologies that are, respectively, connected and disconnected. 

We recall this well-known result (\cite{Fied1},\cite{Fied2}) highlighting properties of the laplacians associated to connected and disconnected topologies.

 \begin{proposition} \label{TMconnect} For all $i=1,\ldots, N_{\rm t}$, let $\Lambda_i = \{\lambda_1(L_i), \ldots, \lambda_{N_{\rm a}}(L_i)\}$ be the eigenvalues of $L_i$ ordered with increasing real part. The following holds: 
 \begin{itemize}
 \item if ${\cal T}_i \in {\cal T}_{\rm c}$ then $\lambda_1(L_i)=0$ and $\mbox{Re}\lambda_m(L_i)>0$ for  $m=2,\ldots N_{\rm a}$;
 \item if ${\cal T}_i \in {\cal T}_{\rm nc}$ then there exists a $\rho_i  \in [2,  N_{\rm a}]$ such that $\lambda_m(L_i)=0$ for  $m=1\,\ldots \rho_i$ and $\mbox{Re}\lambda_m(L_i)>0$  for  $m=\rho_i+1\,\ldots N_{\rm a}$.
 \end{itemize}
  \end{proposition}
 
Each of the $N_{\rm a}$ agents is described by the nonlinear dynamics
   \beeq{ \label{exo1}
   \ba{rcl}
  \dot w_k &=& s(w_k)  + u_k \qquad w_k \in \Real^d\\
  y_k &=& c(w_k)
  \ea
  }
 in which, for each $k=1,\ldots, N_{\rm a}$, $u_k \in \Real$ is the control input, $y_k \in \Real$ is the available measurement, and $s(\cdot)$ and $c(\cdot)$ have the form better specified in Section \ref{SecAss}. Note that we deal with homogenous nonlinear agents, namely $s(\cdot)$ and $c(\cdot)$ do not depend on $k$. 

We look for a decentralised control structure in which the control law of a single agent is taken as
 \beeq{\label{nunu}
 u_k = K\nu_k^i\,, \qquad  
 \nu_k^i= \sum_{j=1}^{N_{\rm a}} a_{kj}^i\, (c(w_j)-c(w_k))
 }
 with $K$ to be designed in such a way that output consensus is reached among the agents. Namely,  for each  initial condition $w_k(0)\in \Real^d$, there is a function $y^\ast : \Real \to \Real$  such that
\[
\lim_{t\to \infty}|y_k(t) - y^\ast(t)|=0\,,
\]
uniformly in the initial conditions, for all $k=1,\ldots, N_{\rm a}$.
It is worth noting that, in the proposed framework, no leader is considered, and only the neighbour's information is available according to the underlying communication topology. Furthermore, local output of single agents rather than a {\em full state} information is assumed to be spread over the network.

The different communication topologies alternates in time by forming an ordered sequence  $\{{\cal T}_i\}_{i=1}^\infty$, with each ${\cal T}_i$ taken in the set $\cal T$.  We denote by $\Delta T_i \geq 0$, $i=1,\ldots, \infty$ the length of the time interval in which the $i$-th communication topology is active. Note that time intervals of zero length are allowed in the proposed framework. By this fact, without loss of generality, we can assume that the topologies alternates in time according to the rule that ${\cal T}_i \in {\cal T}_{\rm c}$ if $i$ is odd and ${\cal T}_i \in {\cal T}_{\rm nc}$ is $i$ is even. As a matter of fact, if two connected (disconnected)  communication topologies occur in a row  we can always "separate" them with a disconnected (connected) topology of zero length without practically changing the networked system dynamics. Note also that we do not assume that connected communication topologies persist for a guaranteed dwell time, namely connected topologies can last for arbitrarily small (indeed also of length zero) time interval.  The kind of result we will prove   (see next Proposition \ref{prop1}) is that consensus is reached if the intervals of time in which connected topologies govern the communication between the agents have a sufficiently long (in the average) duration.  

\section{Main assumptions}\label{SecAss}
The main result  is proved under a certain number of assumptions that are here presented. First, we assume that the nonlinear agent dynamics (\ref{exo1}) are of the form 
      \beeq{\label{exo2}
 s(w) = Sw + B\phi(w)\,, \qquad c(w)=Cw
 }
 with $(S,B,C)$ a triplet of matrices in {\em prime} form, that is $S$ is a shift matrix (all $1$'s on the upper diagonal and all $0$'s elsewhere), $B^T=\left(0 \cdots 0 \;\; 1\right)$ and $C=\left(1 \;\; 0 \cdots 0 \right)$, and the function $\phi(\cdot)$ that is  globally Lipschitz, namely there exist a positive constant $\bar \phi$ such that $|\phi(w)| \leq \bar \phi$ for all $w \in \Real^d$. 
 
 Furthermore, as in \cite{nostro}, we assume that agents (\ref{exo1}) have a robust compact attractor $W \subset \Real^d$, where robustness is characterised in terms of Input-to-State Stability.
\begin{assumption}\label{ass1}
   There exists a compact set  $W\subset \Real^d$ invariant  for (\ref{exo1}) with $u=0$ such that the system
\[
  \dot w = S w + B \phi(w) + u
\]
 is input-to-state stable with respect to $u$ relative to $W$, namely there exist
a class-${\cal K}{\cal L}$ function $\beta(\cdot,\cdot)$ and a class-$\cal K$ function $\gamma(\cdot)$ such that\footnote{Here and in the following we denote by $\|w\|_{W}=\min_{x \in W}\|w-x\|$  the
distance of $w$ from $W$. Furthermore, $w(t,\bar w)$ denotes the solution of (\ref{exo1}) at time $t$ with initial condition $\bar w$ at time $t=0$.
}
 \[
 \|w(t,\bar w)\|_{W} \leq \max\{ \beta(\|\bar w\|_W,t),\; \gamma(\sup_{\tau \in [0,t)}\|u(\tau)\|)\}\,.
 \]
 \end{assumption}
  
  We now fix some restrictions on the communication topologies that might occur in the network. The first assumption asks that the real part of the nontrivial eigenvalues of the  laplacians of the possible topologies are bounded from below by a known constant $\mu$. 
 \begin{assumption} \label{ass2}
There exist a $\mu >0$ such that, for all $i=1,\ldots,N_{\rm t}$ and for all $m=1,\ldots, N_{\rm a}$ such that $\lambda_m(L_i) \neq 0$, the following holds:
\[
\mbox{\rm Re}\lambda_m(L_i) \geq \mu\,.
\]
 \end{assumption} 
 
 It is worth noting that such assumption is automatically fulfilled if the elements $a_{kj}^i$ of the adjacency matrices associated 
 to the  topologies in $\cal T$ range in known intervals (since $\cal T$ is a finite set). In particular, it is automatically fulfilled in case of binary adjacency matrices. 
 We formulate now an assumption about the length of the time intervals in which disconnected topologies are active. The assumption simply asks 
 that such a length is upper bounded by a known constant.  
 
 \begin{assumption} \label{ass3}
There exists a $T_0 >0$ such that for all  $i=1,\ldots,N_{\rm t}$ such that ${\cal T}_i \in {\cal T}_{\rm nc}$ the following holds
\[
\Delta T_i \leq T_0\,.
\]
 \end{assumption} 
  
  The previous conditions are clearly not enough to prove any  asymptotic consensus between the agents since no properties on the time intervals in which the topologies are connected are formulated. In this respect the additional condition under which the main result will be proved asks that the time intervals in which the network is connected last, in the average, sufficiently long. More  precisely, we asks that there exist positive  $\tau\in \Real_{\geq 0}$ and $n_0 \in \mathbb N$ such that, for all possible $n, i_0 \in \mathbb N$ with $i_0$ odd, we have
\beeq{\label{ADWT}
\sum_{i=i_0, \,i=i +2}^{i_0 + 2n} \Delta T_i \geq \tau \, (n-n_0) 
}

The previous condition can be regarded as a average dwell-time condition (see \cite{HespanhaMorse}), with the time $\tau$, in particular, that can
be seen as an average length of the intervals in which the network is connected, and $n_0$ representing the number of "connected" intervals of zero
duration that can occur in a row.  The result formulated in the next proposition, in fact, claims that consensus is achieved if (\ref{ADWT}) is fulfilled for some $n_0$ and $\tau$ with the latter sufficiently large. 

\section{Main Result}\label{SecOS}
 Following \cite{nostro} we choose the vector $K$ as
 \beeq{\label{K}
 K=D_gK_0
 }
  where $D_g={\rm diag}(g,g^2, \ldots, g^d)$ with $g$ is a ``gain" parameter, and
 $K_0$ is chosen as 
 \beeq{\label{K0}K_0=P C^T}
 with $P$ solution of the algebraic Riccati equation
 \beeq{\label{Riccati}
   S P + PS^T - 2 \mu P C^T C P + a I=0
  }
  with $a>0$, $S$ and $C$ as in (\ref{exo2}) and $\mu$ as in Assumption \ref{ass2}.
  
\begin{proposition}\label{prop1}
 Consider the networked control system (\ref{exo1}) controlled by (\ref{nunu}) with $K$ as in (\ref{K}) under the Assumptions \ref{ass1}-\ref{ass3} an with the length of the time interval of connected topology fulfilling the average dwell-time condition (\ref{ADWT}) for some $n_0\geq 1$ and $\tau$.  Then there exist a $\tau^\star$ and $g^\star$ such that for all $\tau \geq \tau^\star$ and $g \geq g^\star$ the compact  invariant set
\beeq{\label{bigw}\ba{l}
\ba{rcl}
{\bf W} &=& \{(w_1,w_2,\ldots,w_{N_{\rm a}})\in W\times W \times \cdots \times W\\
 && \hspace*{2.5cm} \; : \; w_1=w_2=\cdots= w_{N_{\rm a}}\}\ea
\ea
}
 is globally asymptotically stable for the closed-loop network system.  $\triangleleft$
\end{proposition}

\begin{proof}
Consider the closed-loop networked system obtained by controlling the agents (\ref{exo1}), with $s(\cdot)$ and $c(\cdot)$, as in (\ref{exo2}),  with decentralised controllers (\ref{nunu}) when the communication network is governed by the generic $i$-th topology ${\cal T}_i \in \cal T$.  
By bearing in mind the definition of Laplacian, we first observe that the $\nu_k^i$ in (\ref{nunu}) can be re-written as
\[
 \nu_k^i = -\sum_{j=1}^{N_{\rm a}}\, \ell_{kj}^i C w_j\qquad k=1,\ldots, N_{\rm a}\,.
 \]

By defining $\mathbf{w}=\mbox{col}(w_1,\ldots,w_{N_{\rm a}})$ and by bearing in mind the choice of $K$, the networked system can be compactly rewritten as\footnote{Here and in what follows, $A \otimes B$ denotes the Kronecker product of the two matrices $A$ and $B$.} 
 \beeq{\label{bw}
\dot {\bf w} = \left[ (I_{N_a} \otimes S) - (L_i \otimes D_g K_0 C) \right ] {\bf  w}+ (I_{N_a} \otimes B) \Phi({\bf w})
}
where $\Phi({\bf w})=\mbox{col}(\phi(w_1), \ldots, \phi(w_{N_{\rm a}}))$. 

 Consider now the matrix  $T_i\in \Real^{N_{\rm a} \times N_{\rm a}}$  defined as
  \[
   T_i= \left( \ba{c} 1 \quad 0_{1 \times (N_{\rm a}-1)}\\
   \Upsilon_i  \qquad
   \ea \right ) \quad \Upsilon_i := \left( \ba{cc}1_{N_{\rm a}-1} & J_i \ea \right )
  \]
  where ${J_i} \in \mathbb{R}^{(N_{\rm a}-1) \times (N-_{\rm a}1)}$ is a non singular matrix that transforms the matrix \[L_{i \, [2:N_{\rm a}, 2:N_{\rm a}]} - {1}_{N_{\rm a}-1} \,L_{i \, [1,2:N_{\rm a}]}\] into the Jordan form. Note that $T_i$ is non singular with the inverse given by
  \[
   T_i^{-1}= \left( \ba{c} 1 \quad 0_{1 \times (N_{\rm a}-1)}\\
   \Upsilon_i^{-1}  \qquad
   \ea \right ) \quad \Upsilon_i^{-1} := \left( \ba{cc}- J_i^{-1}1_{N_{\rm a}-1} & J_i^{-1} \ea \right )\,.
  \]
  
  Elementary computations show that   
  \beeq{\label{LapTilde}
   \tilde{L}_i = T_i^{-1} L_i T_i = \left( \ba{cc} 0 & L_{i, {12}}\\
    0_{(N_{\rm a}-1) \times 1} & L_{i, {22}}
    \ea \right )
 }
 where $L_{i, {12}}=L_{i \, [1,2:N_{\rm a}]} J_i$ and  $L_{i, {22}} = J_i^{-1} ( L_{i \, [2:N_{\rm a}, 2:N_{\rm a}]} - {1}_{N_{\rm a}-1} \,L_{i \, [1,2:N_{\rm a}]}) J_i$. The latter, by the specific choice of $J_i$, takes the block diagonal form
 \[
  L_{i, {22}} = \mbox{blkdiag}(L_{i,\rm nc}, \, L_{i, \rm c})
 \]
with $L_{i,\rm nc}$ and $ L_{i, \rm c}$ having all the eigenvalues, respectively, in the origin and with positive real part. In particular, having in mind  Proposition \ref{TMconnect} and the fact that 
$\mbox{eig}(L_i) =\mbox{eig}(\tilde L_i) =  \{ 0 \} \cup \mbox{eig}(L_{i,\rm nc}) \cup \mbox{ eig}(L_{i,\rm c})$, it turns out that $L_{i,\rm nc}$ is absent  if ${\cal T}_i \in {\cal T}_{\rm c}$, while the dimension of $L_{i,\rm nc}$ is $\rho_i - 1 \geq 1$ if  ${\cal T}_i \in {\cal T}_{\rm nc}$.
  We consider now the change of variables
  \[
   {\bf w} \quad \mapsto \quad \left ( \ba{c} z_1\\ {\bf z} \ea \right ) = (T_{i}^{-1} \otimes I_d ) {\bf w} \,,
  \]
  with $z_1 \in \Real^d$ and ${\bf z} \in \Real^{(N_{\rm a} -1)d}$. Note that, by definition of $T_i$, $z_1 = w_1$, ${\bf z} = (\Upsilon_i^{-1} \otimes I_d) {\bf w}$ and 
  ${\bf w} = (\Upsilon_i \otimes I_d) {\bf z}$. By using (\ref{LapTilde}), an easy calculation shows that system (\ref{bw}) in the new coordinates reads as
  \[
  \ba{rcl}
   \dot z_1 &=& S z_1 + B \phi(z_1) - (L_{i, {12}} \otimes D_gK_0 C) {\bf z}\\[2mm]
   \dot {\mathbf{z}} &=& \left[ (I_{N_{\rm a}-1} \otimes S) - (L_{i,{22}} \otimes D_g K_0 C) \right ] { \bf z}
   +\Delta \Phi(z_1, {\bf z})
   \ea
  \]
  where
  \[
   \Delta \Phi (z_1, {\bf z}) = (I_{N-1} \otimes B) \left( \ba{ccc} \phi(z_1+z_2) - \phi(z_1)\\
    \phi(z_1+z_3) - \phi(z_1)\\
    \cdots\\
    \phi(z_1+z_N) - \phi(z_1)
   \ea \right )
  \]
  having partitioned ${\bf z} =\mbox{col}(z_2, \ldots, z_{N_{\rm a}})$, with $z_i \in \Real^d$, $i=2,\ldots, N_{\rm a}$. Note that, by the fact that $\phi(\cdot)$ is globally Lipschitz, we have that  $ \|\Delta \Phi (z_1, {\bf z})\| \leq \bar \Phi \|{\bf z}\|$ for all $z_1 \in \Real^d$ and ${\bf z} \in \Real^{(N_{\rm a} - 1)d}$, for some positive constant $\bar \Phi$.
  
 As proposed in \cite{nostro}, we now rescale the variable $\bf z$ in the following way
 \beeq{\nonumber \zeta=(I_{N_{\rm a}-1} \otimes D_g^{-1}){\bf z}} 
by obtaining 

 \[
\begin{array}{rcl}
 \dot z_1 &=& S z_1 + B \phi(z_1) - (L_{i, {12}} \otimes D_gK_0 C) (I_{N_{\rm a}-1} \otimes D_g) {\zeta}\\
\dot{\zeta}
&=& g {H}_i \zeta + \dst \frac{1}{g^d} {\Delta \Phi}(z_1,(I_{N_{\rm a}-1} \otimes D_g)\zeta)
\end{array}
\]
  where $H_{i}=[(I_{N-c} \otimes S)-(L_{i,22} \otimes K_0C)]$. In particular, by taking advantage from the block diagonal structure of $L_{i,22}$ and by bearing in mind Proposition \ref{TMconnect}, the $\zeta$ subsystem can be rewritten as two decoupled systems of the form
  
   \beeq{ \label{gamma}
 \begin{array}{rcl}
 \dot{\zeta}_{\rm nc}&=&g\, H_{i, {\rm nc}} \, \zeta_{\rm nc} + \dst \frac{1}{g^d} {\Delta \Phi}_{\rm nc}(z_1,(I_{\rho_i-1} \otimes D_g)\, \zeta_{\rm nc}) \\
 \dot{\zeta}_{ \rm c}&=&g H_{i, \rm c} \zeta_{\rm c} + \dst \frac{1}{g^d} {\Delta \Phi}_{\rm c}(z_1,(I_{N_{\rm a}-\rho_i} \otimes D_g)\, \zeta_{\rm c})  
  \end{array} 
  }
 where  ${\zeta}_{\rm nc} \in \Real^{(\rho_i-1)d}$,  ${\zeta}_{ \rm c} \in \Real^{(N_{\rm a} - \rho_i)d}$, 
 \[
 \ba{rcl}
 H_{i, {\rm nc}} &=& (I_{\rho_i-1} \otimes S)-(L_{i,\rm nc} \otimes K_0C)\\[2mm]
 H_{i, {\rm c}} &=& (I_{N_{\rm a}-\rho_i} \otimes S)-(L_{i,\rm c} \otimes K_0C)
 \ea
 \]
 and ${\Delta \Phi}_{\rm nc} \in \Real^{\rho_i-1}$, ${\Delta \Phi}_{\rm c} \in \Real^{N_{\rm a}-\rho_i}$ are such that ${\Delta \Phi}(\cdot) = \mbox{col}({\Delta \Phi}_{\rm nc}(\cdot), {\Delta \Phi}_{\rm c}(\cdot))$. We observe that, according to Proposition \ref{TMconnect}, $\rho_i=1$ if the $i$-th topology is connected, while 
$\rho_i\in[2,N_{\rm a}]$ otherwise. Namely,  the ${\zeta}_{\rm nc}$ dynamics in (\ref{gamma}) are absent if ${\cal T}_i \in {\cal T}_{\rm c}$.
 
 In the remaining part of the proof we show that the origin of system (\ref{gamma}) is globally asymptotically stable if the communication topologies $\{{\cal T}_i\}_{i=1}^\infty$ governing the networked system satisfy Assumption 3 and  the average dwell-time condition (\ref{ADWT}) with a $\tau$ sufficiently large.  
 
 To this end we consider the candidate Lyapunov function  
 \beeq{ \label{lyap} 
  V(\zeta)=\zeta^T (D(\ell) \otimes {P}^{-1}) \zeta
 }
  where $P$ is the solution of (\ref{Riccati}) and $D(\ell) = \mbox{diag}(1, \ell, \ell^2, \ldots, \ell^{N_{\rm a} -2})$ with $\ell$ a positive design parameter yet to be fixed. Note that there exist positive constants $\underline \lambda \leq \bar \lambda$, both dependent on $\ell$, such that $\underline \lambda \zeta^T \zeta \leq V \leq \bar \lambda \zeta^T \zeta$ .
  We study the behaviour of $V$ during the time intervals in which the communication topology is connected, when it is disconnected, and when there is transition between two different topologies.
  
 We start by considering the case in which ${\cal T}_i \in {\cal T}_{\rm c}$ ($i$ odd), by first observing that in this case the $\zeta_{\rm nc}$ subsystem in (\ref{gamma}) is absent, namely $H_i = H_{i, \rm c}$.  The  derivative of $V$ along the solutions of (\ref{gamma}) can be thus bounded as 
   \[
   \ba{rcll}
    \dot V &=& 2 \zeta^T (D(\ell) \otimes P^{-1}) [g H_{i} \zeta+ \\&&\qquad \dst {1 \over g^d} \Delta \Phi(z_1,(I_{N_a-1}\otimes D_g)\zeta)]  \\[4mm]
     &\leq &2 \zeta^T (D(\ell) \otimes P^{-1}) g H_{i,\rm c} \zeta  +\\
    && \qquad  \dst {2\over g^d} \bar \Phi  \|D(\ell) \otimes P^{-1}\| \|(I_{N_{\rm a}-1} \otimes D_g)\| \zeta^T \zeta \\[2mm]
    &\leq&  2 \zeta^T (D(\ell) \otimes P^{-1}) g H_{i} \zeta + a_\phi  \zeta^T \zeta \\[2mm]
   \ea
  \]
  where $a_{\phi}$ is a positive constant not dependent on $g$ (provided that the latter is taken $g \geq 1$). To further elaborate $\dot V$ we recall this 
  crucial result (see \cite{nostro}).
   \begin{lemma}
   Let Assumption 2 hold. Then there exist positive constants $a_{\rm c}'$ and $\ell^\star$ such that for all $\ell \geq \ell^\star$ and for all $i$
   \[
      2 \zeta_{\rm c}^T (D(\ell) \otimes P^{-1}) H_{i, \rm c} \, \zeta_{\rm c} \leq - a_{\rm c}' \zeta_{\rm c}^T \zeta_{\rm c}\,.
   \]
  \end{lemma}
  
  Using the previous lemma and taking $g^\star =(a_\phi + a_{\rm c} \bar \lambda)/ a_{\rm c}' $ with $a_{\rm c}$ an arbitrary positive constant ($g^\star \geq 1$ without loss of generality), it is immediately seen that for all $\ell \geq \ell^\star$ and $g \geq g^\star$ we have 
  
    \beeq{\label{dotVconn}
   \ba{l}
    \dot V \leq  -(g a_{\rm c}' - a_\phi) \zeta^T \zeta  \leq -\dst {g a_{\rm c}' - a_\phi \over \bar \lambda} V \leq  
     - \dst {a_{\rm c}} V\,.
   \ea
  }

 We consider now time intervals in which ${\cal T}_i \in {\cal T}_{\rm nc}$ ($i$ even). In this case we have
     \beeq{\label{dotVdisc}
   \ba{l}
    \dot V = 2 \zeta^T (D(\ell) \otimes P^{-1}) [g H_i \zeta + \dst {1 \over g^d} \Delta \Phi(z_1,(I_{N_a-1}\otimes D_g)\zeta)]\\[2mm]
    \leq g \,a_{\rm nc}' \, \zeta^T \zeta+ \dst {2\over g^d} \bar \Phi  \|D(\ell) \otimes P^{-1}\| \|(I_{N_{\rm a}-1} \otimes D_g)\| \zeta^T \zeta \\[2mm]
    \leq (g \,a_{\rm nc}'  + a_\phi)  \zeta^T \zeta = \dst a_{\rm nc} V\\[2mm]
   \ea
  }
  where $a_{\rm nc}' := 2 \| (D(\ell) \otimes P^{-1})  H_{i}\|$, $a_{\phi}$ is the positive constant introduced above, and $a_{\rm nc} = (g a_{\rm c}' + a_{\phi})/\underline \lambda$. 
 
 We now estimate the jump in the value of $V(\zeta)$ when a change in the topology occurs, namely when ${\cal T}_{i+1}$ replaces ${\cal T}_i$. Denoting by $\zeta^+$, ${\bf z}^+$ and ${\bf w}^+$ the "next value" of the state variables $\zeta$, $\bf z$ and $\bf w$ when a jump in the topology occurs, and by bearing in mind the definition of $\zeta$ and $\bf z$, we have ${\bf w}^+=\bf w$ and 
  \beeq{ \label{zeta+}
\ba{rcl}
{\zeta}^+&=&(I_{N_{\rm a}-1} \otimes D_g^{-1}) {\bf z}^+\\[1mm]
 &=& (I_{N_{\rm a}-1} \otimes D_g^{-1}) (\Upsilon_{i+1}^{-1} \otimes I_d){\bf w}^+\\[1mm]
 &=& (I_{N_{\rm a}-1} \otimes D_g^{-1}) (\Upsilon_{i+1}^{-1} \otimes I_d){\bf w}\\[1mm]
 &=& (I_{N_{\rm a}-1} \otimes D_g^{-1}) (\Upsilon_{i+1}^{-1} \otimes I_d)(\Upsilon_{i} \otimes I_d){\bf z}\\[1mm]
 &=& (I_{N_{\rm a}-1} \otimes D_g^{-1}) (\Upsilon_{i+1}^{-1} \otimes I_d)(\Upsilon_{i} \otimes I_d)(I_{N_{\rm a}-1} \otimes D_g) \zeta\\[1mm]
 &=& (\Upsilon_{i+1}^{-1}\Upsilon_i \otimes I_d) \zeta\,.
\ea
}
Hence, by letting
\[
\bar \upsilon= \max_{i,j \in [1,\ldots N_{\rm t}]} \|(\Upsilon_{j}^{-1}\Upsilon_i \otimes I_d) \|\,
\]
we can easily bound the jump of the Lyapunov function when  the topology switches as 
 \beeq{\label{Vjump}
 \ba{rcl}
  V^+&=&\zeta^{+T} (D(\ell) \otimes {P}^{-1}) \zeta^+ \leq {\bar \lambda} \|\zeta^{+}\|^2\\[1mm]
   &=& {\bar \lambda} \|(\Upsilon_{i+1}^{-1}\Upsilon_i \otimes I_d) \zeta\|^2 \leq  {\bar \lambda} \|(\Upsilon_{i+1}^{-1}\Upsilon_i \otimes I_d)\|^2 \|\zeta\|^2\\[1mm]
   &\leq& \dst {\bar \lambda} \bar \upsilon^2 \|\zeta\|^2 \leq   {\bar \lambda \over \underline \lambda} \, \bar \upsilon^2\,  V := a_{\rm j} V\,.
 \ea
 }

 We will continue the analysis by considering the closed-loop networked  system as an hybrid system flowing during the time intervals
in which the communication topology is connected ($i$ odd), and ''instantaneously'' jumping in the intervals in which the topology is disconnected. To this end, let $i$ be odd and let  $t_i$, $t_{i+1}$ be, respectively, the times at which the topology switches from ${\cal T}_i \in {\cal T}_{\rm c}$ to ${\cal T}_{i+1}\in {\cal T}_{\rm nc}$, and from ${\cal T}_{i+1}$ to ${\cal T}_{i+2} \in {\cal T}_{nc}$. By bearing in mind (\ref{Vjump}), (\ref{dotVdisc})
 and Assumption 3 we have that the jump undergone by the Lyapunov function between two connected topologies can be estimated as 
 \[
 \ba{rcl}
 V(t_{i+1}^+) &\leq& a_{\rm j } V(\zeta(t_{i+1}^-)) \leq a_{\rm j } e^{a_{\rm nc} T_0}  V(\zeta(t_{i}^+))\\[1mm]
  & \leq& a_{\rm j } e^{a_{\rm nc} T_0}  a_{\rm j}V(\zeta(t_{i}^-)) = e^{c_{\rm j}} V(\zeta(t_{i}^-))
\ea
 \]
 with $c_{\rm j} :=  a_{\rm nc} T_0 + 2 \ln ( a_{\rm j})$.
We are thus left to study an hybrid system governed by (\ref{dotVconn}) during flows and instantaneously jumping as $V^+ \leq e^{c_{\rm j}} V$, with the length of the flow intervals governed by an average dwell time of the form (\ref{ADWT}). 

 The fact that the time intervals satisfy an average dwell-time condition given by  (\ref{ADWT}) allows one to say (see \cite{CTG08}) that flow and jump times of the hybrid system can be thought of as governed by a clock variable $\varsigma$ flowing according to the differential inclusion $\dot \varsigma \in [0, 1/\tau]$ when $\varsigma \in [0,n_0]$ and jumping as $\varsigma^+ =\varsigma -1$ when $\varsigma \in [1,n_0]$. We thus endow the networked system with the clock variable and study the resulting hybrid system whose Lyapunov function flows and
jumps according to the following rules 
 \[
\ba{l}
 \left . 
 \ba{rcl}
 \dot \varsigma &\in& [0, 1/\tau]\\
 \dot V &\leq& - a_{\rm c} V
 \ea \right \} \quad (\varsigma, V)\in  [0,n_0] \times \Real\\[4mm]
 \left.
 \ba{rcl}
 \varsigma^+ &=& \varsigma -1\\
 V^+ &\leq& e^{c_{\rm j}} V
 \ea \right \} \quad (\varsigma, V)\in  [1,n_0] \times \Real\\ 
\ea
 \]
  
 Specifically, following \cite{CTG08}, let  
\[
{W}(\varsigma,\zeta)=e^{L\varsigma} {V(\zeta)}
\] 
with $L \in $ $(c_{\rm j},\tau a_{\rm c}/ 2)$. 
 During flows, by compactly writing (\ref{gamma}) as $\dot \zeta = F(\zeta,z_1)$,  
 we have that for all  $v  \in \mbox{col}\left([0,1 /\tau]\,, {F}(\zeta, z_1)  \right )$
\[
\ba{rcl}
 \langle \nabla {W}(\tau,\zeta), v \rangle &=& L \,e^{L \varsigma} \, \dot \varsigma \, {V}(\zeta) + e^{L \varsigma} \langle \nabla {V}(\zeta), F(\zeta,z_1) \rangle\\
 &\leq&
 \dst L e^{L \varsigma} {1 \over \tau} {V}(\zeta) - a_{\rm c} e^{L \varsigma} V(\zeta)\\
 &\leq&  \dst - (a_{\rm c} - {L \over \tau} ) {W}(\varsigma, \zeta)\\
 &\leq& - \dst {a_{\rm c} \over 2} {W}(\varsigma,\zeta)
 \ea 
\]
for all $(\varsigma,\zeta, z_1) \in [0,n_0] \times \mathbb{R}^{(N_{\rm a}-1)d} \times \mathbb{R}^{d}$. On the other hand, during jumps, we have that  
\[
\ba{rcl}
 W^+ &=&\dst e^{L\varsigma^+} {V}(\zeta^+) \leq  e^{L (\varsigma -1)} e^{c_{\rm j}} V(\zeta)\\[1mm]
 &=&  \dst e^{-L + c_{\rm j}}e^{L \varsigma} V(\zeta) = 
 \dst e^{-L + c_{\rm j}}W(\varsigma,\zeta)\\[1mm]
 &=& \varepsilon  W(\varsigma,\zeta)
\ea
\]
with $\varepsilon = e^{-L + c_{\rm j}} \in (0,1)$. The Lyapunov function $W(\cdot,\cdot)$ is thus decreasing both during flows and during jumps. This and the fact that $W$ is positive definite with respect to the set $[0,n_0] \times \{0\}$ lead to the conclusion that the set $[0,n_0] \times \{0\}$ is globally asymptotically stable. This,  by taking advantage from Assumption 1, proves the result. 
\end{proof}

\section{Conclusions}

 We addressed the problem of achieving consensus within a network of homogeneous nonlinear agents with switching communication (connected and disconnected) topologies. We have proved that if the switching rule of connected topologies fulfils an average dwell-time condition and the length of time intervals in which disconnected topologies are active is upper bounded by a constant, consensus is achieved.
 The proposed result considers  a network of homogeneous agents. Following the arguments in \cite{nostro}, however,  robust consensus among heterogeneous nonlinear systems exchanging information within a switching network of the kind considered in this paper can be obtained.

\end{document}